\documentclass[11pt,a4paper]{article}
\usepackage{mathrsfs}
\usepackage{amsmath,amsthm,color,graphicx,eepic}
\usepackage{amssymb} \usepackage{verbatim}
\usepackage{dsfont}
\usepackage{graphicx}
\usepackage{subcaption}

\usepackage{hyperref}

\usepackage{mathtools} %coloneqq

\newtheorem{theorem}{Theorem}%[section]

\newtheorem{lemma}[theorem]{Lemma}

\newtheorem{claim}{Claim}
\newtheorem*{claim*}{Claim}

\newcommand{\B}{\mathcal{B}}
\newcommand{\ex}{{\rm ex}}

\newcommand{\floor}[1]{\left\lfloor{#1}\right\rfloor}

\newcommand{\ind}[1]{\mathds{1}_{\text{odd}(#1)}}
\begin{document}

\title{The Connected Bipartite Turán Problem for Long Cycles and Paths}
\author{
Zhen He\thanks{School of Mathematics and Statistics, Beijing Jiaotong University, Beijing, China.} \and
Nika Salia\thanks{Theoretical Computer Science Department, Faculty of Mathematics and Computer Science, Jagiellonian University, Krak\'{o}w, Poland.} \and
Xiutao Zhu\thanks{School of Mathematics, Nanjing University of Aeronautics and Astronautics, Nanjing, China.} \and
}
\date{}
\maketitle

\begin{abstract}
Caro, Patkós, and Tuza initiated a systematic study of the bipartite Turán number for trees, and in particular asked for the extremal number of edges in connected bipartite graphs with prescribed color-class sizes that contain no paths of given lengths.
In this paper, we determine these numbers exactly and describe all corresponding extremal configurations.
Our approach first establishes a more general result for long cycles: we determine the exact structure of all 2-connected bipartite graphs with no cycle of length at least a given constant.
The proof combines Kopylov’s method for long cycles with a strengthened version of Jackson’s classical lemma, in which every extremal configuration is characterized.

To highlight the applicability of our results, we conclude with applications yielding concise proofs of classical theorems on bipartite Turán numbers, notably rederiving the results of Gyárfás, Rousseau, and Schelp for paths and Jackson for long cycles.
\end{abstract}

\section{Introduction}

A central question in extremal combinatorics is to determine, for a given graph~$F$, the largest number of edges in an $n$-vertex graph that contains no subgraph isomorphic to~$F$.  
This maximum, denoted by $\ex(n, F)$, is known as the \emph{Turán number} of~$F$.  
The general asymptotic behavior of $\ex(n, F)$ for non-bipartite graphs was established by Erd\H{o}s, Stone, and Simonovits~\cite{erdos1946structure,erdHos1965limit}, yet the bipartite case remains far from fully understood and continues to attract considerable attention, see the survey by F{\"u}redi and Simonovits~\cite{furedi2013history}.
Even for some of the most natural bipartite graphs, fundamental problems remain open—for instance, the celebrated Erd\H{o}s–Sós~\cite{erd1964extremal} conjecture asserts that every $n$-vertex graph avoiding a given $k$-vertex tree has at most $(k-2)n/2$ edges, a bound known only for special classes of trees and limited values of~$k$.

When the forbidden graph is a path, the exact Turán numbers were determined by Erd\H{o}s and Gallai~\cite{erdHos1959maximal}.  
In fact, they first determined the exact Turán number for the family of long cycles $ C_{\ge \ell}$, from which the result for paths followed immediately.
Originally, we intended to replicate this approach in the bipartite setting—deriving the result for paths as a corollary of the corresponding statement for long cycles.
However, an unexpected phenomenon arises: the extremal numbers for consecutive odd and even paths coincide, while the corresponding extremal constructions differ, rendering this reduction inapplicable for even paths.
Erd\H{o}s and Gallai proved that an $n$-vertex graph containing no $P_k$, a path on $k$ vertices, has at most $\tfrac{1}{2}(k-2)n$ edges, with equality attained by the disjoint union of copies of $K_{k-1}$.  
A natural refinement of this problem, motivated by the disconnectedness of the extremal example, is to impose connectivity on the host graph.  
This connected variant, typically denoted by $\ex_c(n, F)$, asks for the maximum number of edges in an $n$-vertex \emph{connected} graph without a copy of~$F$.  
For paths, the connected version was determined by Kopylov~\cite{kopylov1977maximal}.
Kopylov obtained the extremal number for $2$-connected graphs forbidding long cycles and used this to derive the exact connected Turán number for paths.  
Our approach follows the same underlying philosophy and closely parallels Kopylov’s methods: we first determine the extremal number and characterize all $2$-connected bipartite graphs without long cycles, and then deduce the corresponding result for connected bipartite graphs forbidding long paths.
Further refinements and stability results for connected settings were later obtained by Balister, Győri, Lehel, and Schelp~\cite{balister2008connected} and by Füredi, Kostochka, Luo, and Verstraëte~\cite{furedi2018stability,furedi2016stability}.

Caro, Patkós, and Tuza~\cite{caro2024connected} initiated a systematic study of the connected Turán number for trees, determining its exact value for all trees on at most six vertices and introducing several general lower-bound constructions based on structural parameters of~$F$.
Building on this, Jiang, Liu, and Salia~\cite{jiang2024connectivity} resolved a question posed in that work by establishing that, asymptotically, the connectivity constraint can at most halve the conjectured extremal number for trees.

Most recently, Caro, Patkós, and Tuza~\cite{caro2025bipartite} extended this line of study to the bipartite setting, introducing the bipartite Turán number $\ex_b(a,b, F)$ and its connected analogue $\ex_{b,c}(a,b, F)$.
For a bipartite graph $F$ and integers $1 \le a \le b$, they defined $\ex_b(a,b,F)$ as the maximum number of edges in an $F$-free bipartite graph with part sizes $a$ and $b$, and $\ex_{b,c}(a,b,F)$ as the corresponding connected version. 
They established general bounds relating these quantities to the classical Turán number, determined exact values for several families of trees, and initiated the study of extremal structures within connected bipartite graphs.

The study of long cycles in bipartite graphs with prescribed part sizes has a rich history.  
Jackson~\cite{jackson1981cycles,jackson1985long} proved a series of sharp results on the existence of long cycles in $2$-connected bipartite graphs, see Theorem~\ref{thm:jackson}, later extended by Jackson and Li~\cite{jackson1994hamilton} to Hamilton cycles in regular bipartite graphs.  
Their work provided powerful degree-based criteria that serve as a tool for most modern approaches to bipartite cycle problems.
The first sharp results on paths in bipartite graphs were obtained by Gyárfás, Rousseau, and Schelp~\cite{gyarfas1984extremal}, who determined the exact extremal bounds for paths in bipartite graphs under various degree and part-size constraints, see Theorem~\ref{thm:GRS}.  
More recently, Kostochka, Luo, and Zirlin~\cite{kostochka2020super} extended Jackson’s results to the hypergraph setting and proved superpancyclicity theorems for bipartite graphs and uniform hypergraphs, see also~\cite{kostochka2022longest}.

Our work determines the exact connected bipartite Turán numbers for paths and long cycles, together with a full structural characterization of the extremal graphs.
Studying connected Turán problems in the bipartite setting for paths and long cycles without fixing the color-class sizes would yield little new insight, as the resulting bounds closely mirror those of the general (non-bipartite) case and follow almost directly from known extremal results.
In contrast, prescribing the sizes of the two color classes makes the problem substantially richer: the extremal constructions must balance the bipartition while maintaining connectivity, often forcing one part to contain many vertices of degree one or two—a seemingly inefficient and highly non-uniform configuration that nevertheless turns out to be optimal.

Independent of our work, Bonamy, Leclere, and Picavet~\cite{Bonamy2025} investigated the same problem for connected bipartite graphs avoiding long paths. They determined the exact maximum number of edges in a bipartite connected graph as a function of the length of its longest path and the sizes of its color classes, and also discussed a particular family of trees, referred to as brooms. Since both studies were developed concurrently and address overlapping questions, we agreed to submit our papers independently, without exchanging drafts. Consequently, while the main objectives coincide, the precise scope and techniques of their work were not known to us at the time of submission.

In general, forbidding a single cycle is a much harder extremal problem—often even the correct order of magnitude of the extremal function is unknown.  
However, when the forbidden cycle is sufficiently long, the problem becomes more tractable, as excluding one very long cycle is asymptotically close to forbidding all long cycles, and the corresponding extremal numbers nearly coincide.  
In this direction, Győri~\cite{gyiiri1997ch} conjectured and partially proved results on bipartite graphs forbidding very long cycles in terms of the part sizes.  
This conjecture was recently solved by Li and Ning~\cite{li2021exact}, who determined the exact bipartite Turán numbers for large even cycles, thereby refining Jackson’s earlier bounds.

\section{Main Results and Structure of the Paper}

In this section, we present our results on connected bipartite Turán numbers for paths and long cycles.
As discussed in the introduction, these problems extend the classical theorems of Erd\H{o}s–Gallai and Kopylov to the bipartite setting with prescribed color-class sizes, and answer questions of Caro, Patkós, and Tuza.

We first address the case of long cycles and obtain exact results for bipartite graphs without cycles of length at least a given constant.
Since cycles are $2$-connected, requiring connectivity adds no further restriction, and the natural setting for this problem is that of $2$-connected bipartite graphs.
Accordingly, our first main theorem determines the precise extremal number and characterizes all such $2$-connected bipartite graphs.

\begin{theorem}\label{Thm:2_connected_bipartite_Cycle_free}
Let $G$ be a $2$-connected $C_{\ge 2\ell}$-free bipartite graph with bipartition classes of sizes $a$ and $b$, where $b \ge a \ge \ell\geq 4$. Then
\[
|E(G)| \le \frac{\ell}{2}(a + b)
           + \left(\frac{\ell}{2} - 2\right)(b - a - 4) - 4.
\]
Equality holds precisely for the unique bipartite graph $B_{2}(a,b,2\ell)$ obtained from $K_{\ell - 2,\, b}$ by adding $a-(\ell-2)$ independent vertices, each adjacent to the same two vertices in the part of size $b$, see Figure\ref{Fig:B_2}.
\end{theorem}

Note that the conditions on the parameters are natural.  
If $a < \ell$, then the maximum is trivially attained by the complete bipartite graph $K_{a,b}$.  
On the other hand, if $\ell \le 3$, there exists no $2$-connected bipartite graph without a cycle of length at least $4$, and the only $2$-connected bipartite graph that is $C_{\ge 6}$-free is $K_{2,b}$.

\begin{figure}[!ht]
    \centering
    \begin{subfigure}{0.4\textwidth}
    \includegraphics[width=\textwidth]{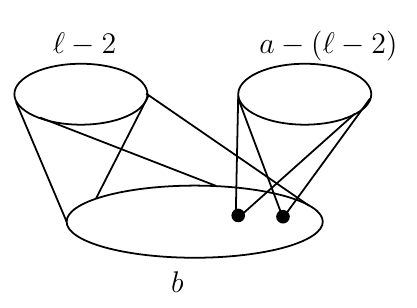}
    \caption{$B_2(a,b,2\ell)$.}
    \label{Fig:B_2}
    \end{subfigure}
    \hfill
\begin{subfigure}{0.4\textwidth}
    \centering
    \includegraphics[width=\textwidth]{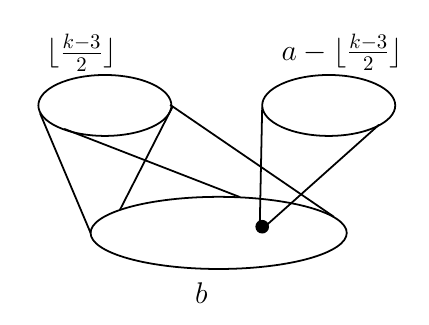}
    \caption{$B_1(a,b,k)$.}
    \label{Fig:B_1}
    \end{subfigure}
    \caption{Extremal bipartite graphs $B_2(a,b,2\ell)$ (2-connected) and $B_1(a,b,k)$ (connected).}
    \label{Fig:Extremal_2-connected_Const}

\end{figure}
 
As an immediate consequence, we obtain the exact connected bipartite Turán numbers for paths, thereby resolving a question posed by Caro, Patkós, and Tuza.

\begin{theorem}\label{Thm:Connected_Path_Bipartite}
Let $G$ be a connected $P_{k}$-free bipartite graph with bipartition classes of sizes $a$ and $b$, where $b \ge a \ge \tfrac{k}{2} \geq 4$.  
Then
\[
|E(G)| \le \big\lfloor\tfrac{k-3}{2}\big\rfloor b + \big(a - \big\lfloor\tfrac{k-3}{2}\big\rfloor\big).
\]
Equality holds precisely for the unique bipartite graph $B_{1}(a,b,k)$ obtained from $K_{\lfloor (k-3)/2 \rfloor,\, b}$ by adding $a - \lfloor (k-3)/2 \rfloor$ independent vertices, each adjacent to the same vertex in the part of size~$b$ when $k$ is odd, see Figure~\ref{Fig:B_1}.  
If $k$ is even, these degree-one vertices may be incident with arbitrary vertices of the color class of size~$b$, we denote this class of graphs by $\B_{1}(a,b,k)$ .
\end{theorem}

Note that the conditions on the parameters are again natural.
Interestingly, the extremal number in Theorem~\ref{Thm:Connected_Path_Bipartite} takes the same value for paths of consecutive odd and even lengths, a phenomenon that does not occur in the classical (non-bipartite) setting.

In Subsection~\ref{Sub_Sec:Sketch}, we present a brief overview of the proof strategy.
 In  Subsection~\ref{Sub_Sec_Jack}, we present our main technical tool — Jackson’s Lemma — which can be viewed as an extension of Dirac’s classical lemma for $2$-connected graphs extended for bipartite settings.
Informally, it states that a maximal path in a graph,
 guarantees the existence of a long cycle, whose length is at least the degree sum of the path’s endpoints or equal to the path length itself.
Jackson extended this result to bipartite graphs.
However, in the bipartite setting, an additional negative constant term appears in the lower bound on the cycle length, depending on the length of the path, which complicates the analysis.
To address this, we revisit Jackson’s original argument and give a characterization of extremal configurations.
In the Subsection~\ref{Sec:Proof}, we prove Theorem~\ref{Thm:2_connected_bipartite_Cycle_free} and in Subsection~\ref{Sub:Corollary} deduce Theorem~\ref{Thm:Connected_Path_Bipartite} as its corollary. 
In Section~\ref{sec:Applications}, we apply our main theorems to derive brief proofs of known results on the bipartite Turán numbers for paths and long cycles, and we also announce a forthcoming work that applies our result to determine a lower bound of a parameter.

\subsection{Overview of the Proof Strategy}\label{Sub_Sec:Sketch}

We prove Theorem~\ref{Thm:2_connected_bipartite_Cycle_free} by considering an edge-maximal $C_{\ge 2\ell}$-free bipartite graph with the largest possible number of edges — that is, a graph in which adding any additional edge creates a long cycle.
We identify a large subgraph with high minimum degree and show that, unless it is a complete bipartite graph, the desired bound follows directly from Jackson’s lemma or from the structural properties we derive.
If the subgraph happens to be complete, we construct another large subgraph with slightly smaller minimum degree and apply Jackson’s lemma once more to obtain the required contradiction.

For Theorem~\ref{Thm:Connected_Path_Bipartite}, we begin with the standard argument based on the similarity of the two constructions: removing a single vertex from a $2$-connected extremal graph yields the extremal connected graph without long paths.
However, this reasoning works only when the forbidden path has an odd length.
When the forbidden path has even length, we adopt a different approach: decompose the graph into its $2$-connected blocks, analyze their structures separately, and apply induction on the number of vertices.
We note that the proof for the even case also implies the result for the odd case.
Nevertheless, since the direct argument for the odd case is shorter and somewhat relevant, we include it separately.

\subsection{Notions and Tools}\label{Sub_Sec_Jack}

Throughout the paper, all graphs are simple, finite, and undirected.  
For a graph $G$, we denote by $V(G)$ and $E(G)$ its vertex and edge sets, respectively.  
The degree of a vertex $v$ in $G$ is denoted by $d_G(v)$, and $\delta(G)$ stands for the minimum degree of~$G$.  
For $U \subseteq V(G)$, we write $G[U]$ for the subgraph induced by the vertices from~$U$.  

A graph $G$ is called \emph{bipartite} if its vertex set can be partitioned into two independent sets $A$ and $B$; in this case, we write $G(A,B)$, and refer to $A$ and $B$ as its \emph{color classes}.  
A graph is \emph{$2$-connected} if it is connected and remains connected after deleting any single vertex.  
For a positive integer $\ell$, we write $C_{\ge 2\ell}$ for the family of all cycles of length at least $2\ell$, and $P_k$ for the path on $k$ vertices.

We use $N_G(v)$ to denote the neighborhood of a vertex $v$ in $G$, and for a maximal path $P=v_1v_2\dots v_m$, we write  
\[
N^+(v_m) := \{v_{i+1} : v_i \in N(v_m),\, i < m\}, \qquad
N^-(v_1) := \{v_{i-1} : v_i \in N(v_1),\, i > 1\}.
\]
These notations are consistent with those used in Jackson’s lemma and its subsequent refinements.

We conclude this subsection with the statement and a brief proof sketch of Jackson’s lemma, which serves as the central tool in our arguments.

\begin{lemma}[Jackson~\cite{jackson1985long}]\label{Lemma_Jackson}
    Let $G$ be a $2$-connected bipartite graph, and let $P_m$ be a maximal path in $G$, with $u,~v$ terminal vertices.
    Then $G$ contains a cycle of length at least
  \[
  \min\left\{m-\ind{m},2\left(d(u)+d(v)-1-\ind{m}\right)\right\}.
  \]
\end{lemma}
\begin{proof}[Sketch of the proof]
This is a standard lemma with a well-known proof~\cite{jackson1985long}; therefore, we only sketch it here to illustrate the main idea and clarify the conditions under which the minimum is achieved.  

Let $P_m := v_1v_2 \dots v_m$ be a maximal path in $G$.  
Let $i$ be the largest and $j$ the smallest indices such that $v_1v_i \in E(G)$ and $v_m v_j \in E(G)$.

If $i < j$, then there is a cycle containing all the vertices from 
$N(v_1)$, $N^+(v_m)$, and $N(v_m)$.  
The sets $N(v_1)$, $N^+(v_m)$, and $N(v_m)$ are pairwise disjoint, and $N(v_1)$ lies in the same color class as one of them.  
Hence the cycle has length at least $2(d(v_1) + d(v_m))$, since 
$|N(v_1)| = d(v_1)$ and $|N^+(v_m)| = |N(v_m)| = d(v_m)$.  

If $i = j$, then $m$ is odd, and by the previous argument, the cycle has length at least 
$2(d(v_1) + d(v_m)) - 2$, since $N(v_1)$ and $N(v_m)$ share one vertex.

If $i > j$, let us choose a pair $(i', j')$, $i'>j'$ that minimizes $i' - j'$ among $v_1v_{i'} \in E(G)$ and $v_m v_{j'} \in E(G)$. 
If $i' - j'\in \{1,2\}$ then there is a cycle of length $m-\ind{m}$.  
Otherwise, $N(v_1)$ is disjoint from both $N^+(v_m)$ and $N^{++}(v_m)$, note that $|N^{++}(v_m)|=d(v_m)-1$, as $v_mv_{m-1}\in E(G)$.
There exists a cycle containing the disjoint monochromatic sets $N(v_1)$ and 
$N^{+}(v_m) \setminus \{v_{j'+1}\}$ if $m$ is even, or 
$N^{++}(v_m) \setminus \{v_{j'+2}\}$ if $m$ is odd.  
Thus, the length of this cycle is at least $2(d(v_1) + d(v_m) - 1-\ind{m})$.  
The vertices $v_x$, for $x \in [j'+1, i'-1]$, belong to none of the sets 
$N(v_1)$, $N^{+}(v_m)$, or $N^{++}(v_m)$, and the cycle contains all vertices of $P_m$ except those.  
Hence, either there exists a cycle of length $2(d(v_1) + d(v_m) - 1)$, or there is no other choice for the minimal pair $(i', j')$ with the same properties.  
In summary, if no longer cycle exists than $2(d(v_1) + d(v_m) - 2)$, then the path length $m$ is odd and
\[
N(v_1) = \{v_2, \dots, v_j\} \cup \{v_{i'}, \dots, v_i\},
\qquad
N(v_m) = \{v_j, \dots, v_{j'}\} \cup \{v_i, \dots, v_{m-1}\},
\]
see also the Figure~\ref{Fig:Extremal_Config}.

\begin{figure}[!ht]
    \centering
    \includegraphics[width=0.6\textwidth]{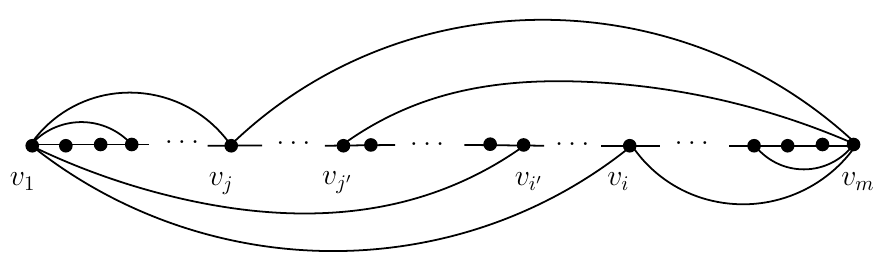}
    \caption{Schematic structure of the extremal configuration arising in Jackson’s Lemma,
showing the neighborhoods of the terminal vertices of a longest path in a 2-connected bipartite graph with circumference of length $2(d(v_1) + d(v_m) - 2)$.}

    \label{Fig:Extremal_Config}
\end{figure}

\end{proof}

\section{Proofs of the Main Results}\label{sec:proofs}

In this section, we prove our two main theorems.
In Subsection~\ref{Sec:Proof}, we establish the extremal result for $2$-connected bipartite graphs without long cycles (Theorem~\ref{Thm:2_connected_bipartite_Cycle_free}).
In Subsection~\ref{Sub:Corollary}, we then deduce the corresponding result for connected bipartite graphs without long paths (Theorem~\ref{Thm:Connected_Path_Bipartite}) as a consequence.

\subsection{Two Connected Bipartite Graphs Without Long Cycles}\label{Sec:Proof}
\begin{proof}[Proof of Theorem~\ref{Thm:2_connected_bipartite_Cycle_free}]
Assume that
\[
|E(G)| \ge \frac{\ell}{2}(a + b) + \left(\frac{\ell}{2} - 2\right)(b - a - 4) - 4.
\]  
We further assume, without loss of generality, that $G$ is an edge-maximal bipartite graph with respect to being $C_{\ge2\ell}$-free, that is, adding any additional edge to $G$ creates a cycle of length at least $2\ell$.

The average degree of $G$ satisfies
\[
\overline{d}
:= \frac{\tfrac{\ell}{2}(a + b)
      + \left(\tfrac{\ell}{2} - 2\right)(b - a - 4) - 4}{a + b}
> \frac{\ell}{2} - 1.
\]
Let $G'$ be the subgraph of $G$ obtained by repeatedly removing vertices
of degree at most $\tfrac{\ell}{2}$.
Note first that $G'$ is nonempty for every even $\ell \ge 4$.
Indeed, suppose to the contrary that $G'$ becomes empty.
Then, at some stage of the deletion process, one color class has size $\tfrac{\ell}{2}$,
while the other has size $x > \tfrac{\ell}{2}$.
At that moment, we obtain
\[
|E(G)|
\le x \cdot \frac{\ell}{2}
   + \frac{\ell}{2}\bigl(a + b - \tfrac{\ell}{2} - x\bigr)
\le \frac{\ell}{2}(a + b)
   + \Bigl(\frac{\ell}{2} - 2\Bigr)(b - a - 4) - 4.
\]
The second inequality can only be tight when $\ell = 4$; for every  $\ell > 4$, this contradicts our assumption on $|E(G)|$.
When $\ell = 4$, either we again obtain a contradiction, or at the last stage the graph consists of a copy of $K_{2,x}$, $x \ge 3$, and every vertex removed previously had degree exactly~$2$.

In the latter situation, let $v_1$ and $v_2$ be the two vertices of the smaller color class of this $K_{2,x}$, and let $v_3$ be the last vertex deleted from this color class.
In the opposite color class (of size $x \ge 3$), let $u_1$ and $u_2$ be the common neighbours of $\{v_1, v_2, v_3\}$, and let $u_3$ be another vertex in that class.
By $2$-connectivity, there is no edge in $G\bigl[V(G)\setminus \{v_1, v_2, v_3, u_1, u_2, u_3\}\bigr]$,
otherwise we could find a cycle of length at least $8$, which is a contradiction. Thus we have $G \cong B_2(a,b,8)$.

If $\ell$ is odd, $G'$ is not empty by the same argument in a simpler form, as the corresponding inequality above is then strict.

Therefore, we may assume $G'$ is not empty and 
\[
\delta(G') \ge \left\lfloor \tfrac{\ell + 2}{2} \right\rfloor .
\]
Let $P_m := v_1v_2 \dots v_m$ be a longest path in $G$ whose terminal vertices $v_1, v_m$ belong to $V(G')$.
Note that if $G'$ is not a complete bipartite graph, then by the edge-maximality of $G$, we have $m \ge 2\ell$.
From here, we distinguish two cases: the first when $m \ge 2\ell$, and the second when $G'$ is a complete bipartite graph.

\textbf{Case 1.} $m \ge 2\ell$.

From Lemma~\ref{Lemma_Jackson}, we conclude that either $G$ contains a cycle of length at least $2\ell$, yielding a contradiction, or $m$  and $\ell$ are odd integers.  
In the latter case, by the sketch of the proof of Lemma~\ref{Lemma_Jackson}, we have
\[
N(v_1) = \{v_2, \dots, v_j\} \cup \{v_{i'}, \dots, v_i\},
\qquad
N\textbf{}(v_m) = \{v_j, \dots, v_{j'}\} \cup \{v_i, \dots, v_{m-1}\},
\]
Note that $j > 2$, since otherwise we obtain a cycle 
$v_m v_2 v_3 \dots v_m$, of length at least $2\ell$, a contradiction.  
Observe also that $v_2 \in V(G')$.

In this paragraph, we show that 
$N_{G'}(v_2)\cap V(P_m)\subseteq N_{G'}^-(v_1)\setminus \{v_{i'-1}\}$ 
and that $N_{G'}(v)=N_{G'}(v_1)$ for all $v\in N_{G'}(v_2)$.
First, observe that if 
$N_{G'}(v_2)\cap V(P_m)\subseteq N_{G'}^-(v_1)\setminus \{v_{i'-1}\}$, 
then, since $d(v_1)=\tfrac{\ell+1}{2}\le d(v_2)$, 
there exists a vertex $v_1' \in N_{G'}(v_2)\setminus V(P_m)$.  
This gives another longest path 
$P'_m := v_1'v_2\dots v_m$ in $G$ 
whose terminal vertices belong to $V(G')$.  
Hence, either we reach a contradiction, or $N_{G'}(v_1') = N_{G'}(v_1)$.  
Moreover, every vertex in 
$V(P_m)\cap N_{G'}^-(v_1)\setminus \{v_{i'-1}\}$ 
has neighbors along $P_m$ that are also neighbors of $v_1$, 
so these vertices can be exchanged with $v_1'$.  
Thus, without loss of generality, we may assume that 
$N_{G'}(v_2)\cap V(P_m) = N_{G'}^-(v_1)\setminus \{v_{i'-1}\}$.
The following arguments show that indeed 
$N_{G'}(v_2)\cap V(P_m)\subseteq N_{G'}^-(v_1)\setminus \{v_{i'-1}\}$.
First, if $N_{G'}(v_2)\cap N_{G'}^+(v_m)\neq \emptyset$, 
then $G$ contains a cycle of length $m-1$, a contradiction.  
If $v_2v_{i'-1}\in E(G)$, then 
\[
v_2v_3\dots v_{j-2}v_1v_i v_{i+1}\dots v_m v_j v_{j+1}\dots v_{i-1}v_2
\]
is a cycle of length at least $2\ell$, again a contradiction. 

The subgraph $G[N_{G'}(v_2)\cup N_{G'}(v_1)]$ is complete bipartite, 
as obtained in the previous paragraph.
Applying the same argument symmetrically to $v_m$ and $v_{m-1}$, we see that 
$G[N_{G'}(v_{m-1})\cup N_{G'}(v_m)]$ is also complete bipartite.  
Hence, $G$ contains two copies of 
$K_{\frac{\ell + 1}{2},\, \frac{\ell + 1}{2}}$ 
sharing two vertices $v_i$ and $v_j$ in the same color class.  
By the $2$-connectivity and edge-maximality of $G$, it follows that 
\[
G \cong K_{\ell - 1,\, b},
\]
contradicting the assumption $a \ge \ell$.

\textbf{Case 2.} $G' \cong K_{c,d}$, $\frac{\ell+1}{2}\leq c \leq d$.

We have $c < \ell - 1$, otherwise we can find a copy of $C_{2\ell}$ since $a \ge \ell$ and $G$ is $2$-connected.  
Let $G''$ be the subgraph of $G$ obtained by repeatedly removing vertices of degree at most $\ell - c$.  
Clearly, $G' \subseteq G'' \subseteq G$, and $\delta(G'') \ge \ell - c + 1$.  
We now consider two cases: $G'' \cong G'$ and $G'$ being a proper subgraph of $G''$.

If $G'' \cong G'$, then
\[
|E(G)| \le cd + (a + b - c - d)(\ell - c)
      \le (\ell - 2)b + (a - \ell + 2)\cdot 2 \le |E(G)|.
\]
The second inequality follows by a straightforward substitution of $d = b$ and $c = \ell - 2$.  
Equality in this inequality holds only when $d = b$, $c = \ell - 2$, and thus $G \cong B_{2}(a,b,2\ell)$.

If $G' \ne G''$, then by the edge-maximality of $G$ there exists a maximal path $P_m := v_1 v_2 \dots v_m$ in $G$ of length $m \ge 2\ell$, 
where $v_1 \in V(G')$ and $v_m \in V(G'')$.  
By Lemma~\ref{Lemma_Jackson}, either $G$ contains a cycle of length $2\ell$, which is impossible, 
or $m$ is odd and
\[
N_{G'}(v_1) = \{v_2, \dots, v_j\} \cup \{v_{i'}, \dots, v_i\},
\qquad
N_{G''}(v_m) = \{v_j, \dots, v_{j'}\} \cup \{v_i, \dots, v_{m-1}\}.
\]
As in the previous case, the same reasoning shows that 
$G[N_{G'}(v_2)\cup N_{G'}(v_1)]$ and 
$G[N_{G''}(v_{m-1})\cup N_{G''}(v_m)]$ 
are complete bipartite graphs sharing two vertices $v_i$ and $v_j$ in the same color class.  
Since 
\[
|N_{G'}(v_2)\cup N_{G''}(v_{m-1})| \ge \ell + 1 
\quad\text{and}\quad 
|N_{G'}(v_1)\cup N_{G''}(v_m)| \ge \ell - 1,
\]
and $G$ is a $2$-connected $C_{\ge 2\ell}$-free graph, we conclude that
\[
G \cong K_{\ell - 1,\, b},
\]
contradicting the assumption that $a \ge \ell$.
\end{proof}

\subsection{Connected Bipartite Graphs Without Long Paths}\label{Sub:Corollary}

\begin{proof}[Proof of Theorem~\ref{Thm:Connected_Path_Bipartite}]
We provide separate proofs for even and odd values of $k$.

\textbf{Case 1.}  $k = 2\ell + 1$ for some integer $\ell \ge 4$.

Remove all vertices of degree one from the color class $B$, and let the resulting connected graph be $G'$ with color class sizes $a$ and $b'$.  
We will show that Theorem~\ref{Thm:Connected_Path_Bipartite} holds for $G'$.  
This immediately implies $G' \cong G$, by a straightforward argument considering the cases $a \le b'$ and $a > b'$.  
Hence, without loss of generality, we assume $G' \cong G$ and proceed to prove Theorem~\ref{Thm:Connected_Path_Bipartite} for $G$.

Let $G''$ be the graph obtained from $G$ by adding a universal vertex to the color class $B$, adjacent to all vertices in the class $A$.  
The graph $G''$ is $2$-connected and contains no cycle of length at least $2\ell + 2$, since $G$ is connected and $P_{2\ell+1}$-free.  
By Theorem~\ref{Thm:2_connected_bipartite_Cycle_free}, we obtain
\[
|E(G)| \le |E(G'')| - a 
      \le (\ell - 1)(b + 1) + 2(a - \ell + 1) - a 
      =\left\lfloor \frac{k - 3}{2} \right\rfloor b + \bigl(a - \left\lfloor \tfrac{k - 3}{2} \right\rfloor \bigr).
\]
Moreover, equality in the second inequality holds when 
$G'' \cong B_2(a, b + 1, 2\ell + 2)$, 
which implies $G \cong B_1(a, b, \ell + 1)$.

\textbf{Case 2.} $k = 2\ell$ for some integer $\ell \ge 4$. 

Assume that $G$ is a $P_{2\ell}$-free bipartite graph with color classes of sizes $\ell \le a \le b$ and 
\[
|E(G)| \ge (\ell - 2)b + a - (\ell - 2).
\]

\begin{claim}\label{claim:2l-2-free}
$G$ is $C_{\ge 2\ell-2}$-free.
\end{claim} 
\begin{proof}
Assume, to the contrary, that $G$ contains a cycle $C = x_1x_2\cdots x_{2\ell-2}$ of length $2\ell - 2$.  
Since $G$ is connected and $P_{2\ell}$-free, there can be no edge in $G - C$; otherwise, such an edge together with $C$ would yield a path on at least $2\ell$ vertices by connectivity.  

Let $u \in A \setminus V(C)$ and $v \in B \setminus V(C)$ be vertices of maximum degree in their respective color classes.  
Then $(N(u)^+ \cup N(u)^-) \cap N(v) = \emptyset$, implying $d(u) + d(v) \le \ell - 2$.  
Furthermore, no edge can join $N(u)^+$ to $N(v)^+$ or $N(u)^-$ to $N(v)^-$, as this would create a copy of $P_{2\ell}$.  
Since $d(u) \le \ell - 2$, we also have $N(u)^- \setminus N(u)^+ \neq \emptyset$.  

It follows that the number of missing edges within $C$ is at least $d(u)d(v) + 1$, and therefore
\[
|E(G)| \le (\ell - 1)^2 - (d(u)d(v) + 1) + d(u)(a - \ell + 1) + d(v)(b - \ell + 1).
\]
The expression on the right is maximized when $d(u) = 1$ and $d(v) = \ell - 3$, yielding
\[
|E(G)| \le (\ell - 3)b + a + 1 < b(\ell - 2) + a - (\ell - 2),
\]
a contradiction.
\end{proof}

If $\ell \le 4$, the statement is immediate, since every $2$-connected block either contains a $C_6$ or is isomorphic to $K_{t,2}$.  
Hence, we may assume $\ell \ge 5$.  

If $G$ is $2$-connected, then, as $(\ell - 2)b + a - (\ell - 2) \ge (\ell - 3)b + 2(a - (\ell - 3))$, Theorem~\ref{Thm:2_connected_bipartite_Cycle_free} implies the existence of a cycle of length at least $2\ell - 2$, contradicting Claim~\ref{claim:2l-2-free}.

Let $G'$ be a minimal subgraph of $G$ with color classes of sizes satisfying 
$\ell - 1 \le a'$ and $a', \ell \le b'$, and such that 
\[
|E(G')| \ge (\ell - 2)b' + a' - (\ell - 2).
\]
The graph $G'$ is obtained from $G$ by successively removing all vertices of degree one and, whenever possible, a pair of vertices—one from each color class whose degree sum is at most $\ell - 1$, as long as the resulting graph remains connected.

If $a' = \ell - 1 < b'$, then if $G'$ is $2$-connected, Theorem~\ref{Thm:2_connected_bipartite_Cycle_free} implies that $G'$ contains a cycle of length at least $2\ell - 2$, since $(\ell - 2)b' + 1 \ge (\ell - 3)b' + 2 \cdot 2,$
contradicting Claim~\ref{claim:2l-2-free}.
If $G'$ is not $2$-connected, then the maximum number of edges in a bipartite graph that is not $2$-connected with color class sizes $\ell - 1$ and $b' > \ell$ is attained by the graph $K_{\ell - 2, b'}$ with an additional pendant edge.
Hence $G' \cong H' \in \mathcal{B}_1(a', b', 2\ell)$, which in turn implies that $G \cong H \in \mathcal{B}_1(a, b, 2\ell)$.

If $a' = b' = \ell$, then if $G'$ is $2$-connected, Theorem~\ref{Thm:2_connected_bipartite_Cycle_free} again guarantees a cycle of length at least $2\ell - 2$, since $(\ell - 2)b' + 2 \ge (\ell - 3)b' + 2 \cdot 3$, a contradiction. 
If $G'$ contains a vertex of degree one, we may remove this vertex and apply the previous argument. 
Otherwise, the number of edges in $G'$ satisfies $|E(G')| \le (\ell - 2)(\ell - 1) + 4 < (\ell - 2)\ell + 2,$ which contradicts the assumed lower bound.

From this point onward, we may assume that $\delta(G') \ge 2$ and that no pair of vertices from opposite color classes has a degree sum at most $\ell - 1$ whose simultaneous removal leaves $G'$ connected.

Let the longest cycle $C_{2x}$ in $G'$ lie in the block $B_1$, for some $x \le \ell - 2$.  
Note that $x\geq \floor{\frac{\ell+1}{2}}$, otherwise by Theorem~\ref{thm:jackson} we have
\[|E(G')|\leq \left(a+b- 2\floor{\frac{\ell+1}{2}}+3\right)\floor{\frac{\ell-1}{2}} <(\ell-2)b+a-(\ell-2)\]
a contradiction.
All remaining blocks are trivially $C_{2(\ell - x + 1)}$-free.  If all remaining blocks are $C_{2(\ell - x)}$-free, then applying Theorem~\ref{Thm:2_connected_bipartite_Cycle_free} to each block yields
\begin{equation}\label{Eq:bound}
|E(G')|\leq x\bigl(b' - (\ell - x - 2)\bigr) + (\ell - x - 1)(a' - x) 
   \le (\ell - 2)b' + a' - (\ell - 2).
\end{equation}
Equality in~\eqref{Eq:bound} occurs only when $x = \ell - 2$, in which case $G'$ coincides with $K_{\ell - 2,\, b'}$ with pendant vertices attached.  Trivially, we have $G \cong  H\in \B_1(a,b,2\ell)$.
For all $x < \ell - 2$, the inequality is strict, giving a contradiction.
The intuition behind~\eqref{Eq:bound} is that each $2$-connected block satisfies the bound from Theorem~\ref{Thm:2_connected_bipartite_Cycle_free}, and merging two blocks increases the total number of edges whenever their larger color classes lie on the same side of the bipartition.  

We may assume, otherwise the proof is complete, that $G'$ contains two cycles of lengths $2x$ and $2(\ell - x)$ sharing a single cut vertex. 
Let us refer to the two color classes of $G'$ as the \emph{red} and \emph{blue} classes, and assume that the shared cut vertex  $v_r$ lies in the red class.  
There exist paths with $2\ell - 1$ vertices whose terminal vertices lie in the blue class and whose interiors pass through vertices of both cycles, traversing the cut vertex that connects them.
One of the paths terminates at a vertex $v_b$ which is a consecutive vertex of $v_r$ along the cycle of length $2(\ell - x)$. 
It is straightforward to observe that $d_{G'}(v_b) \le \ell - x$, and since all its neighbors lie on the cycle—and therefore within the same block, $v_b$ cannot be a cut vertex; otherwise, $G$ would contain a copy of a $P_{2\ell}$, a contradiction.

Now, add a new vertex in the blue class adjacent to all the vertices of the red class.  
The resulting graph is $2$-connected since there are no degree one vertices in $G'$.
By Theorem~\ref{Thm:2_connected_bipartite_Cycle_free}, there is a cycle of length at least $2\ell$.  
Consequently, the original graph $G'$ must contain a path $P^r_{2\ell - 1}$ vertices with both endpoints in the red class. Note that the terminal vertices are not cut vertices.   

The path $P^r_{2\ell - 1}$ must contain at least one vertex from the cycle $C_{2x}$ distinct from $v_r$.
Otherwise, by combining all vertices of $C_{2x}$ with at least $\ell$ vertices from $P^r_{2\ell - 1}$, we would obtain a path on $2\ell$ vertices, since $x \ge \left\lfloor \tfrac{\ell + 1}{2} \right\rfloor$.
Note that if $P^r_{2\ell - 1}$ passes through $v_r$, then $v_r$ cannot be its middle vertex when $\ell$ is even, as it belongs to the same color class as the terminal vertices of $P^r_{2\ell - 1}$.

Furthermore, $P^r_{2\ell - 1}$ does not contain any vertex from the block containing $v_b$, except possibly $v_r$.
Otherwise, by a parity argument, one of the subpaths of $P^r_{2\ell - 1}$ starting at $v_r$ together with a subpath of either $C_{2x}$ or $C_{2\ell - 2x}$ starting at $v_r$ would contain at least $2\ell$ vertices, a contradiction.

If at least one terminal vertex of $P^r_{2\ell - 1}$ has degree at most $x - 1$, then this vertex together with $v_b$ form a pair of vertices in different color classes and distinct $2$-connected components whose degree sum is at most $\ell - 1$, contradicting the defining property of $G'$, since removal of these two vertices maintains $G'$ connected.

If $P^r_{2\ell - 1}$ lies entirely within a single $2$-connected block, then by Lemma~\ref{Lemma_Jackson}, $G$ contains a cycle of length at least $\min\{2\ell - 2, 4x - 4\} > 2x$, a contradiction.
Finally, if this is not the case, then there exists a path of length at least $2x + 1$ from an endpoint of $P^r_{2\ell - 1}$ to $v_r$.  
Indeed, since the terminal vertex of $P^r_{2\ell - 1}$ lies outside the block $B_1$ containing $C_{2x}$ and has degree at least $x$, the subpath leading toward $B_1$ must traverse at least $2x$ vertices before reaching the cut vertex connecting to $B_1$.  
As this entry point is distinct from $v_r$, the resulting path includes at least one additional vertex within $B_1$, namely $v_r$.  
Consequently, we obtain a path in $G$ containing at least $2\ell$ vertices, contradicting the assumption that $G$ is $P_{2\ell}$-free.

\end{proof}

\section{Applications of the Main Theorem}\label{sec:Applications}
In this section, we show that Theorems~\ref{Thm:2_connected_bipartite_Cycle_free} and~\ref{Thm:Connected_Path_Bipartite} are not only useful for further developments but also provide short and unified proofs of several classical results.
In particular, we rederive the theorems of Gyárfás, Rousseau, and Schelp~\cite{gyarfas1984extremal}, who determined the exact bipartite Turán numbers for paths without requiring connectivity, by applying Theorem~\ref{Thm:Connected_Path_Bipartite}.
We also note that Jackson~\cite{jackson1985long} determined the bipartite Turán number for long cycles, a result that can likewise be recovered as a corollary of our Theorem~\ref{Thm:2_connected_bipartite_Cycle_free}.

Furthermore, our results form a basis for new results in this area.
In a forthcoming work currently in preparation~\cite{bai2025}, Bai, Li, Patkós, Salia, and Yu apply our results for paths to derive new lower bounds on the maximum ratio between the classical and connected bipartite Turán numbers for general trees.

\begin{theorem}[Gyárfás, Rousseau, and Schelp~\cite{gyarfas1984extremal}]\label{thm:GRS}
Let $a \le b$, then
\[
  \ex_b(a,b, P_{2\ell+2}) =
  \begin{cases}
     ab, & \text{if } a \le \ell, \\[4pt]
     b\ell, & \text{if } \ell + 1 \le a \le 2\ell, \\[4pt]
     (a + b - 2\ell)\ell, & \text{if } a \ge 2\ell,
  \end{cases}
\]
and
\[
  \ex_b(a,b, P_{2\ell+3}) =
  \begin{cases}
     ab, & \text{if } a \le \ell \text{ or } a = b = \ell + 1, \\[4pt]
     a + (b - 1)\ell, & \text{if } \ell + 1 \le a < 2(\ell + 1) \text{ and } b \neq \ell + 1, \\[4pt]
     2(\ell + 1)^2, & \text{if } a = b = 2(\ell + 1), \\[4pt]
     (a + b - 2\ell)\ell, & \text{if } a \ge 2(\ell + 1) \text{ and } b \neq 2(\ell + 1).
  \end{cases}
\]
\end{theorem}

\begin{proof}[Proof using Theorem~\ref{Thm:Connected_Path_Bipartite}]
We present the argument for $\ex_b(a,b, P_{2\ell+3})$, as the even case is analogous.  
Let $G(A,B) = G_1(A_1,B_1) \cup \cdots \cup G_t(A_t,B_t)$ be an extremal $P_{2\ell+3}$-free bipartite graph with color classes of sizes $a$ and $b$, where each $G_i$ is a connected component.  
If $a \le \ell$ or $a = b = \ell + 1$, the statement is trivial.  
If $G$ is connected, then by Theorem~\ref{Thm:Connected_Path_Bipartite},
\[
e(G) \le b\ell + a - \ell,
\]
as required.  
Hence, we may assume $G$ has at least two components.

Since $G$ is extremal, each component $G_i$ has the same structural type as either $B_1(a_i, b_i, 2\ell + 3)$ or $K_{x,y}$ with $x \le \ell + 1$.  

If for some $i<j$ we have  $|A_i| = a_i \le |B_i| = b_i$ and $|A_j| = a_j \le |B_j| = b_j$, then replacing $G_i \cup G_j$ by a single copy of $B_1(a_i + a_j, b_i + b_j, 2\ell + 3)$ increases the number of edges unless $G_i = G_j = K_{\ell + 1, \ell + 1}$.  
If a third component $G_k(A_k,B_k)$ exists, then replacing $G_i \cup G_j \cup G_k$ by $B_1(2(\ell + 1) + a_k, 2(\ell + 1) + b_k, 2\ell + 3)$ again increases $|E(G)|$, a contradiction.  
Hence $G \cong 2K_{\ell + 1, \ell + 1}$ in this case.

If $|A_i| = a_i \le |B_i| = b_i$ but $|A_j| = a_j > |B_j| = b_j$, then $G_i$ and $G_j$ again have the structure of either $B_1(a,b,2\ell+3)$ or $K_{x,y}$ with $x \le \ell + 1$.  
If $G_i$ has a leaf in $A_i$, we can transfer it to $A_j$ in $G_j$, increasing the edge count—a contradiction.  
Hence, both $G_i$ and $G_j$ must be complete bipartite.  
Without loss of generality, let $a_i + a_j \le b_i + b_j$.  
We can then replace $G_i \cup G_j$ by either $K_{a_i + a_j, b_i + b_j}$ or $B_1(a_i + a_j, b_i + b_j, 2\ell + 3)$, depending on whether $a_i + a_j \le \ell$ or $\ell + 1 \le a_i + a_j < 2(\ell + 1)$.  
If $a_i + a_j \ge 2(\ell + 1)$, then replacing $G_i \cup G_j$ by 
$K_{\ell, b_i + b_j - \ell} \cup K_{a_i + a_j - \ell, \ell}$ also increases $|E(G)|$ unless 
$G_i \cup G_j$ already has this form.  
In that final case, no third component can exist, as it would again permit a denser construction.  
This completes the proof.
\end{proof}

\medskip

We also note that Jackson~\cite{jackson1985long} determined the bipartite Turán number for long cycles.

\begin{theorem}[Jackson~\cite{jackson1985long}]\label{thm:jackson}
Let $a \le b$. Then
\[
\ex_b(a,b, C_{\ge 2\ell}) =
\begin{cases}
  (b - 1)(\ell - 1) + a, & \text{if } a \le 2\ell - 2, \\[4pt]
  (a + b - 2\ell + 3)(\ell - 1), & \text{if } a \ge 2\ell.
\end{cases}
\]
\end{theorem}

The extremal examples in Jackson’s theorem are not $2$-connected. Analyzing the structure of each block via Theorem~\ref{Thm:2_connected_bipartite_Cycle_free} yields a concise alternative proof of this result.  
As the reasoning closely parallels the arguments above, we omit the details.

\section*{Acknowledgments}

We thank Binlong Li for drawing our attention to Jackson’s lemma.
In our preliminary drafts, we believed the lemma was a new result, and we are grateful for his helpful remark clarifying its earlier appearance in the literature.

We thank Balázs Patkós for communicating with Marthe Bonamy, Théotime Leclere, and Timothé Picavet about our overlapping results, and we are grateful to them for the collegiality in allowing simultaneous submission of our manuscripts~\cite{Bonamy2025}.

Salia further thanks his coauthors for their kind invitation to China, which made this collaboration possible.

The research of He was supported by the Beijing Natural Science Foundation, grant 1244047, and the National Natural Science Foundation of China, grant 12401445. 
The research of Salia was supported by the National Science Centre grant 2021/42/E/ST1/00193.
The author Zhu is supported by NSFC under grant 12401445, Basic Research Program of Jiangsu Province(BK20241361), Jiangsu Funding Program for Excellent Postdoctoral Talent(2024ZB179).

\bibliographystyle{abbrv}
\bibliography{references.bib}

\end{document}